\documentclass[12pt,fleqn]{amsart}

\usepackage{enumerate}
\usepackage{etex}
\reserveinserts{28}
\usepackage[latin1]{inputenc}
\usepackage{amsmath}
\usepackage{amsfonts}
\usepackage{amssymb}
\usepackage{amsmath,amssymb,amsthm,amsfonts,graphicx,amscd} 
\usepackage{mathtools}
\usepackage{color}                    
\usepackage{hyperref}                 
\usepackage[T1]{fontenc}

\newtheorem{THEO}{Theorem}

\newtheorem{CORO}[THEO]{Corollary}
\newtheorem{LEMM}[THEO]{Lemma}

\newtheorem{DEFI}[THEO]{Definition}

\DeclarePairedDelimiterX{\norm}[1]{\lVert}{\rVert}{#1}

\def\({\left(}
\def\){\right)}

\def\R{\mathbb{ R}}

\def\rto{\rightarrow}

\def\CC{\mathcal{C}}

\def\EE{\mathcal{E}}

\def\1{\textbf{1}}

\def\spn{\operatorname{span}}

\def\w{\omega}

\title[]{Banach spaces with weak*-sequential dual ball}

\author[G.\ Mart\'{\i}nez-Cervantes]{Gonzalo Mart\'{\i}nez-Cervantes}
\address{Departamento de Matem\'{a}ticas\\
	Facultad de Matem\'{a}ticas\\ Universidad de Murcia\\ 30100 Espinardo, Murcia\\
	Spain}
\email{gonzalo.martinez2@um.es}


\thanks{The author was partially supported by  the research project 19275/PI/14 funded by Fundaci\'{o}n S\'{e}neca - Agencia de Ciencia y Tecnolog\'{i}a de la Regi\'{o}n de Murcia within the framework of PCTIRM 2011-2014 and by Ministerio de Econom\'{i}a y Competitividad and FEDER (project MTM2014-54182-P)
}

\begin{document}

\begin{abstract}
A topological space is said to be sequential if every sequentially closed subspace is closed. We consider Banach spaces with weak*-sequential dual ball. In particular, we show that if $X$ is a Banach space with weak*-sequentially compact dual ball and $Y \subset X$ is a subspace such that $Y$ and $X/Y$ have weak*-sequential dual ball, then $X$ has weak*-sequential dual ball. As an application we obtain that the Johnson-Lindenstrauss space $JL_2$ and $C(K)$ for $K$ scattered compact space of countable height are examples of Banach spaces with weak*-sequential dual ball, answering in this way a question of A. Plichko. 
\end{abstract}

\maketitle

\section{Introduction}

All topological spaces considered in this paper are Hausdorff. The symbol $\w^\ast$ denotes the weak* topology of the corresponding Banach space.
A topological space $T$ is said to be \textit{sequentially compact} if every sequence in $T$ contains a convergent subsequence. Moreover, $T$ is said to be \textit{Fr\'{e}chet-Urysohn} (FU for short) if for every subspace $F$ of $T$, every point in the closure of $F$ is the limit of a sequence in $F$. Every FU compact space is sequentially compact.
A Banach space with weak*-FU dual ball is said to have \textit{weak*-angelic dual}. Some examples of Banach spaces with weak*-angelic dual are WCG Banach spaces and, in general, WLD Banach spaces. On the other hand, every weak Asplund Banach space and every Banach space without copies of $\ell_1$ in the dual have weak*-sequentially compact dual ball \cite[Chapter XIII]{Diest}.

In this paper we are going to focus on sequential spaces, which is a generalization of the FU property. If $T$ is a topological space and $F$ is a subspace of $T$, the \textit{sequential closure} of $F$ is the set of all limits of sequences in $F$. $F$ is said to be \textit{sequentially closed} if it coincides with its sequential closure. 
A topological space is said to be \textit{sequential} if any sequentially closed subspace is closed. Thus, every FU space is sequential.
Another natural generalization of the FU property is countable tightness. A topological space $T$ is said to have \textit{countable tightness} if for every subspace $F$ of $T$, every point in the closure of $F$ is in the closure of a countable subspace of $F$. It can be proved that every sequential space has countable tightness. However, whether the converse implication on the class of compact spaces is true is known as the Moore-Mrowka Problem and it is undecidable in ZFC \cite{Ba88}.
Therefore, for a compact space $K$, we have the following implications:
$$
\begin{aligned}
K \mbox{ is FU} \Longrightarrow & ~K \mbox{ is sequential} & \hspace{-2cm} \Longrightarrow  K \mbox{ is sequentially compact} \\
 & \hspace{1cm} \Downarrow \\ 
 &K \mbox{ has countable tightness} 
\end{aligned}
$$

In \cite{Pli15} A. Plichko asked whether every Banach space with weak*-sequential dual ball has weak*-angelic dual. In the next section we prove the following theorem,
which is applied to prove that the Johnson-Lindenstrauss space $JL_2$ provides a negative answer to Plichko's question:
\begin{THEO}
	\label{TeoBase}	
	Let $X$ be a Banach space with weak*-sequentially compact dual ball.
	Let $Y \subset X$ be a subspace with weak*-sequential dual ball with sequential order $\leq \gamma_1$ and such that $X/Y$ has weak*-sequential dual ball with sequential order $\leq \gamma_2$. Then $X$ has weak*-sequential dual ball with sequential order $\leq \gamma_1+\gamma_2$.
\end{THEO}

One of the properties studied by A. Plichko in \cite{Pli15} is property $\EE$ of Efremov.
A Banach space $X$ is said to have\textit{ property $\EE$} if every point in the weak*-closure of any \emph{convex} subset $C \subset B_{X^\ast}$ is the weak*-limit of a sequence in $C$.
We say that $X$ has \emph{property $\EE'$} if every weak*-sequentially closed \emph{convex} set in the dual ball is weak*-closed.
Thus, if $X$ has weak*-angelic dual then it has property $\EE$ and if $X$ has weak*-sequential dual ball then $X$ has property $\EE'$. We also provide a convex version of Theorem \ref{TeoBase} (see Theorem \ref{TeoConv}).

Another related Banach space properties are Mazur property and property (C). A Banach space $X$ has \textit{Mazur property} if every $x^{\ast \ast} \in X^{\ast \ast}$ which is weak*-sequentially continuous on $X^\ast$ is weak*-continuous and, therefore, $x^{\ast \ast} \in X$.
Notice that if a topological space $T$ is sequential then any sequentially continuous function $f:T \rto \R$ is continuous. Thus, it follows from the Banach-Dieudonn\'e Theorem that every Banach space with weak*-sequential dual ball has the Mazur property. Moreover, property $\EE'$ also implies Mazur property.

A Banach space $X$ has \textit{property (C)} of Corson if and only if every point in the closure of $C$ is in the weak*-closure of a countable subset of $C$ for every convex set $C$ in $B_{X^\ast}$ (this characterization of property (C) is due to R. Pol \cite{Pol80}).

Thus, we have the following implications among these Banach space properties: 

$$
\begin{aligned}
\mbox{weak*-angelic dual} & \Rightarrow \mbox{weak*-sequential dual ball}  \Rightarrow \mbox{weak*-sequentially compact dual ball} \\
\Downarrow ~~~~ & \hspace{3cm} \Downarrow \\
\mbox{property } \EE & \hspace{0.5cm} \Longrightarrow \hspace{0.5cm} \mbox{ property } \EE' \hspace{0.5cm} \Longrightarrow \hspace{0.5cm} \mbox{ property } (C) \\ 
& \hspace{3cm} \Downarrow \\
 & \hspace{2cm} \mbox{Mazur property}
\end{aligned}
$$
 
Notice that $\CC([0,\w_1])$ has weak*-sequentially compact dual ball but it is not weak*-sequential.    
Moreover, $\ell_1({\w_1})$ has the Mazur property \cite[Section 5]{Edg79} but it does not have property $(C)$. 

In \cite{PY00} it is asked whether  property (C) implies property $\EE$.
J.T. Moore in an unpublished paper and C. Brech in her PhD Thesis \cite{Brech08} provided a negative answer under some additional consistent axioms, but the question is still open in ZFC. Notice that the convex version of Plichko's question is whether property $\EE'$ implies property $\EE$. A negative answer to this question would provide an example of a Banach space with property (C) not having property $\EE$.

In \cite[Lemma 2.5]{FKK13} it is proved that the dual ball of $\CC(K)$ does not contain a copy of $\w_1+1=[0,\w_1]$ when $K$ is a scattered compact space of finite height satisfying some properties. 
It is also proved in \cite{Kap86} that $\CC(K)$ has the Mazur property whenever $K$ is a scattered compact space of countable height. We generalize these results by proving that $\CC(K)$ has weak*-sequential dual ball whenever $K$ is a scattered compact space of countable height (Theorem \ref{CoroC(K)}).

\section{Banach spaces with weak*-sequential dual ball}

\begin{DEFI}
Let $T$ be a topological space and $F$ a subspace of $T$. 
For any $\alpha \leq \w_1$ we define $S_\alpha(F)$ the $\alpha$th sequential closure of $F$  by induction on $\alpha$: $S_0(F)=F$, $S_{\alpha+1}(F)$ is the sequential closure of $S_\alpha(F)$ for every $\alpha< \w_1$ and $S_\alpha(F)=\bigcup_{\beta < \alpha} S_\beta (F)$ if $\alpha $ is a limit ordinal. 
\end{DEFI}

Notice that $S_{\w_1}(F)$ is sequentially closed for every subspace $F$. Thus, a topological space $T$ is sequential if and only if $S_{\w_1}(F)=\overline{F}$ for every subspace $F$ of $T$.
We say that $T$ has \textit{sequential order $\alpha$} if $S_{\alpha}(F)=\overline{F}$ for every subspace $F$ of $T$ and for every $\beta < \alpha$ there exists $F$ with $S_{\beta}(F) \neq \overline{F}$.
Therefore, a topological space $T$ is sequential with sequential order $\leq 1$ if and only if it is FU.
We will use the following Lemma in the proof of Theorem \ref{TeoBase}:

\begin{LEMM}
\label{LemmAuxClau}
Let $f:K \rto L$ be a continuous function, where $K, L$ are topological spaces and $K$ is sequentially compact. Then, $f(S_\alpha(F))=S_\alpha(f(F))$ for every $F\subset K$ and every ordinal $\alpha$.
\end{LEMM}
\begin{proof}
The inclusion $f(S_\alpha(F)) \subset S_\alpha(f(F))$ follows from the continuity of $f$.

We prove the other inclusion by induction on $\alpha$.
The case $\alpha=0$ is immediate.
Suppose $\alpha=1$. Take $s \in S_1(f(F))$. Then, there exists a sequence $t_n$ in $F$ such that $f(t_n)$ converges to $s$. Since $K$ is sequentially compact, without loss of generality we may suppose $t_n$ is converging to some point $t$. Then, it follows from the continuity of $f$ that $f(t)=s$. Thus, $s \in f(S_1(F))$.

Now suppose the result true for every $\beta < \alpha$ and $\alpha \geq 2$. If $\alpha$ is a limit ordinal then
$$ f(S_\alpha(F))= f( \bigcup_{\beta < \alpha}S_\beta(F))=\bigcup_{\beta < \alpha}f(S_\beta(F))=\bigcup_{\beta < \alpha}S_\beta(f(F))=S_\alpha (f(F)).$$

If $\alpha=\beta+1$ is a successor ordinal then
$$f(S_\alpha(F))=f(S_1(S_\beta(F)))=S_1(f(S_\beta(F)))=S_1(S_\beta(f(F)))=S_\alpha(f(F)).$$
	
\end{proof}

\textit{Proof of Theorem 1.}
It is enough to prove that if $F\subset B_{X^\ast}$ and $0 \in \overline{F}^{\w^\ast}$ then $0 \in S_{\gamma_1+\gamma_2}(F)$.
Let $R: X^\ast \rto Y^\ast$ be the restriction operator. 
For each finite set $A \subset X$ and each $\varepsilon >0$, define
$$ F_{A, \varepsilon} = \lbrace x^\ast \in F: |x^\ast (x)| \leq \varepsilon \mbox{ for all } x \in A \rbrace.$$ 
Since $R$ is weak*-weak* continuous and $0 \in \overline{F_{A,\varepsilon}}^{\w^\ast}$, we have that $$0 \in \overline{R(F_{A,\varepsilon})}^{\w^\ast}=S_{\gamma_1}(R(F_{A, \varepsilon}))=R(S_{\gamma_1}(F_{A, \varepsilon})),$$
where the last equality follows from Lemma \ref{LemmAuxClau}.

Thus, for every finite set $A\subset X$ and every $\varepsilon >0$ we can take $x_{A,\varepsilon}^\ast \in S_{\gamma_1}(F_{A, \varepsilon})$ such that $R(x_{A,\varepsilon}^\ast)=0$.

Therefore, $0 \in \overline{G}^{\w^\ast}$, where
$$ G:= \lbrace x_{A, \varepsilon}^\ast : A \subset X \mbox{ finite, } \varepsilon >0 \rbrace \subset Y^\perp \cap B_{X^\ast}.$$

Note that $(Y^\perp \cap B_{X^\ast}, \w^\ast)$ is homeomorphic to the dual ball of $(X/Y)^\ast$. Hence $$0 \in S_{\gamma_2}(G) \subset S_{\gamma_2}( S_{\gamma_1}(F)) = S_{\gamma_1+ \gamma_2}(F).$$
\hfill\ensuremath{\square}

If $(x_n)_{n<\w}$ is a sequence in a Banach space, we say that $(y_k)_{k<\w}$ is a \textit{convex block subsequence} of $(x_n)_{n<\w}$ if there is a sequence $(I_k)_{k<\w}$ of subsets of $\w$  with $\max(I_k)< \min(I_{k+1})$ and a sequence $a_n \in [0,1]$ with $\sum_{n \in I_k} a_n = 1$ for every $k<\w$ such that
$y_k = \sum_{n \in I_k} a_n x_n$.
A Banach space $X$ is said to have \textit{weak*-convex block compact dual ball} if every bounded sequence in $X^\ast$ has a weak*-convergent convex block subsequence.
Every Banach space containing no isomorphic copies of $\ell_1$ has weak*-convex block compact dual ball \cite{Bou79}. Therefore, every WPG Banach space (i.e. every Banach space with a linearly dense weakly precompact set) also has weak*-convex block compact dual ball.	

For any ordinal $\gamma \leq \w_1$, we say that $X$ has \textit{property $\EE(\alpha)$} if $S_\alpha(C)=C$ for every convex subset $C$ in $(B_{X^\ast}, \w^\ast)$. Thus, property $\EE$ is property $\EE(1)$ and property $\EE'$ is property $\EE(\w_1)$.
The proof of the following theorem is an immediate adaptation of the proof of Lemma \ref{LemmAuxClau} and Theorem \ref{TeoBase}. 

\begin{THEO}
\label{TeoConv}
Let $X$ be a Banach space with weak*-convex block compact dual ball.
Let $Y \subset X$ be a subspace with property $\EE(\gamma_1)$ such that $X/Y$ has property $\EE(\gamma_2)$. Then $X$ has property $\EE(\gamma_1+\gamma_2)$.
\end{THEO}

\begin{THEO}
Let $X$ be a Banach space and $(X_n)_{n < \w}$ an increasing sequence of subpaces with $X= \overline{\bigcup_{n < \w} X_n}$. 
Suppose that each $X_n$ has weak*-sequential dual ball with sequential order $\alpha_n$.  Then  $X$ has weak*-sequential dual ball with sequential order 
$\leq \alpha +1$, where $\alpha:= \sup \lbrace \alpha_n :n<\w \rbrace$. 
\end{THEO}

\begin{proof}
Set $R_n: X^\ast \rto X_n^\ast$ the restriction operator for every $n < \w$.
Since the countable product of sequentially compact spaces is sequentially compact and $(B_{X^\ast}, \w^\ast)$ is homeomorphic to a subspace of $\prod (B_{X_n^\ast}, \w^\ast)$, it follows that $X$ has weak*-sequentially compact dual ball.
In order to prove the theorem, it is enough to prove that if $F \subset B_{X^\ast}$ and $0 \in \overline{F}^{\w^\ast}$ then $0 \in S_{\alpha+1}(F)$.	
Since $B_{X^\ast}$ is weak*-sequentially compact, we have that $0 \in \overline{R_n(F)}^{\w^\ast} = S_{\alpha}(R_n(F))=  R_n(S_{\alpha}(F))$ for every $n < \w$, where the last equality follows from Lemma \ref{LemmAuxClau}.
Thus, we can take a sequence $x_n^\ast \in S_{\alpha}(F)$ such that $R_n(x_n^\ast)=0$.
Now there exists some subsequence of $x_n^\ast$ converging to a point $x^\ast \in S_{\alpha+1}(F)$. Since $R_n(x^\ast)=0$ for every $n<\w$, we conclude that $x^\ast=0$.
\end{proof}

\begin{CORO}
\label{CoroIncreasing}
	Let $X$ be a Banach space and $(X_\alpha)_{\alpha < \gamma}$ an increasing sequence of subpaces with $X= \overline{\bigcup_{\alpha < \gamma} X_\alpha}$, where $\gamma$ is a countable limit ordinal. 
	Suppose that each $X_\alpha$ has weak*-sequential dual ball with sequential order $\leq \theta_\alpha$. Then  $X$ has weak*-sequential dual ball with sequential order 
	$\leq \theta+1$ where $\theta:= \sup \lbrace \theta_\alpha : \alpha< \gamma \rbrace $.
\end{CORO}

The next theorem follows from combining Theorem \ref{TeoBase} and Corollary \ref{CoroIncreasing}	:

\begin{THEO}
\label{Thm}

Let $\gamma$ be a countable ordinal, $X_\gamma$ a Banach space and $(X_\alpha)_{\alpha \leq \gamma}$  an increasing sequence of subpaces of $X_\gamma$ such that:
\begin{enumerate}
	\item $X_0$ has weak*-sequential dual ball with sequential order $\leq \theta$;
	\item each quotient $X_{\alpha+1}/X_{\alpha}$ has weak*-angelic dual;
	\item $X_\alpha = \overline{\bigcup_{\beta < \alpha} X_\beta}$ if $\alpha$ is a limit ordinal;
	\item $X_\gamma$ has weak*-sequentially compact dual ball.
\end{enumerate}	 
Then each $X_\alpha$ has weak*-sequential dual ball with sequential order
$\leq \theta +\alpha$ if $\alpha < \w$ and sequential order $\leq \theta+\alpha +1$ if $\alpha \geq \omega$.
\end{THEO}
\begin{proof}
It follows from $(4)$ that every $X_\alpha$ has weak*-sequentially compact dual ball. Thus,  
the result for $\alpha < \w$ follows by applying inductively Theorem \ref{TeoBase}.
Suppose $\alpha \geq \w$ and $X_\beta$ has weak*-sequential dual ball with sequential order $\leq \theta+\beta+1$ for every $\beta < \alpha$. If $\alpha$ is a limit ordinal then it follows from (3) and from Corollary \ref{CoroIncreasing} that $X_\alpha$ has weak*-sequential dual ball with sequential order $\leq \sup_{\beta<\alpha}\lbrace \theta+\beta+1 \rbrace +1=\theta+\alpha+1$. If $\alpha$ is a successor ordinal then the result is a consequence of Theorem \ref{TeoBase}.	
\end{proof}

We also have the following convex equivalent version of the previous theorem:

\begin{THEO}	
	Let $\gamma$ be a countable ordinal, $X_\gamma$ a Banach space and $(X_\alpha)_{\alpha \leq \gamma}$ an increasing sequence of subpaces of $X_\gamma$ such that:
	\begin{enumerate}
		\item $X_0$ has property $\EE(\theta)$;
		\item each quotient $X_{\alpha+1}/X_{\alpha}$ has $\EE$;
		\item $X_\alpha = \overline{\bigcup_{\beta < \alpha} X_\beta}$ if $\alpha$ is a limit ordinal;
		\item $X_\gamma$ has weak*-convex block compact dual ball.
	\end{enumerate}	 
	Then each $X_\alpha$ has property $\EE(\theta+\alpha)$ if $\alpha < \w$ and  property $\EE( \theta+\alpha +1)$ if $\alpha \geq \omega$.
\end{THEO}

\section{Applications}

As an application of Theorem \ref{TeoBase}, we obtain that the Johnson-Lindenstrauss space $JL_2$ has weak*-sequential dual ball. Let us recall the definition of $JL_2$:

Let $\lbrace N_r: r \in \Gamma \rbrace$ be an uncountable maximal almost disjoint family of infinite subsets of $\w$. For each $N_r$, $\chi_{N_r} \in \ell_\infty$ denotes the characteristic function of $N_r$.
The Johnson-Lindenstrauss space $JL_2$ is defined as the completion of $\spn \( c_0 \cup \lbrace \chi_{N_r}: r \in \Gamma \rbrace \) \subset \ell_\infty $ with respect to the norm:
$$ \norm[\bigg]{x+  \sum_{1 \leq i \leq k} a_i \chi_{N_{r_i}}}= \max \bigg\lbrace \norm[\bigg]{x+  \sum_{1 \leq i \leq k} a_i \chi_{N_{r_i}} }_\infty, \bigg({\sum_{1 \leq i \leq k} |a_i|^2}\bigg)^{\frac{1}{2}} \bigg\rbrace ,$$
where $x \in c_0$ and $\| \cdot \|_\infty$ is the supremum norm in $\ell_\infty$.
If we just consider the supremum norm in the definition then we obtain the space $JL_0$.
We refer the reader to \cite{JL74} for more information about these spaces.

\begin{THEO}
\label{JL}
The Johnson-Lindenstrauss space $JL_2$ has weak*-sequential dual ball with sequential order $2$.
\end{THEO}
\begin{proof}
We use the following results proved in \cite{JL74}:
\begin{enumerate}
	\item [(i)] $JL_2$ has an equivalent Fr\'{e}chet differentiable norm;
	\item [(ii)] $JL_2/c_0$ is isometric to $\ell_2(\Gamma)$.
\end{enumerate}	

It follows from (i) that $JL_2$ has weak*-sequentially compact dual ball (cf. \cite{HS80}).
It follows from (ii) and Theorem \ref{TeoBase} that $JL_2$ has weak*-sequential dual ball with sequential order $\leq 2$.
Since $JL_2$ does not have weak*-angelic dual (cf. \cite[Proposition 5.12]{Edg79}) we have that $JL_2$ has weak*-sequential dual ball with sequential order 2.

\end{proof}

Theorem \ref{JL} provides an example of a Banach space with weak*-sequential dual ball which does not have weakly* angelic dual, answering a question of A. Plichko in \cite{Pli15}.

For a scattered compact space $K$, we denote by $ht(K)$ the height of $K$, i.e. the minimal ordinal $\gamma$ such that the  $\gamma$th Cantor-Bendixson derivative $K^{(\gamma)}$ is discrete. Since every Banach space with weak*-sequential dual ball has the Mazur property, the following theorem improves \cite[Theorem 4.1]{Kap86}:

\begin{THEO}
\label{CoroC(K)}	
If $K$ is an infinite scattered compact space with $ht(K)< \w_1$, then $\CC(K)$ has weak*-sequential dual ball with sequential order $\leq ht(K)$ if $ht(K)<\w$ and with sequential order $\leq ht(K)+1$ if $ht(K) \geq \w$.
\end{THEO}
\begin{proof}
It is well-known that if $K$ is scattered then $\CC(K)$ is Asplund and therefore $B_{C(K)^\ast}$ is weak*-sequentially compact (see, for example, \cite{Yos93}).
Denote by $\lbrace K^{(\alpha)}: \alpha \leq \gamma \rbrace$ the Cantor-Bendixson derivatives of $K$, where $\gamma=ht(K)$.
For every $\alpha \leq \gamma $, set 
$$ X_\alpha = \lbrace f \in \CC(K): f(t)=0 \mbox{ for every } t \in K^{(\alpha)} \rbrace .$$

Since $\CC(K)$ contains a complemented copy of $c_0$, every finite-codimensional subspace of $\CC(K)$ is isomorphic to $\CC(K)$.
Therefore, since $X_\gamma$ is a finite-codimensional subspace of $\CC(K)$, it is isomorphic to $\CC(K)$.
Notice that for every $0 \leq \alpha < \gamma $ we have that $X_{\alpha+1}/X_\alpha$ is isomorphic to $c_0 ( K^{(\alpha)} \setminus K^{(\alpha+1)})$.
Moreover, if $\alpha \leq \gamma$ is a limit ordinal then $\bigcap_{\beta <\alpha} K^{(\beta)}= K^{(\alpha)}$ and therefore
$$ \overline{\bigcup_{\beta< \alpha} X_\beta}= \overline{ \lbrace f \in \CC(K): \exists \beta<\alpha \mbox{ with }f(t)=0 ~~ \forall t \in K^{(\beta)} \rbrace}=X_\alpha.$$
Now the conclusion follows from Theorem \ref{Thm}.
\end{proof}

R. Haydon \cite{Hay78} and K. Kunen \cite{Ku81} constructed under CH a FU compact space $K$ such that $B_{C(K)^\ast}$ does not have countable tightness. Thus, it is not true for a general compact space $K$ that if $K$ is sequential then $B_{C(K)^\ast}$ is sequential. We refer the reader to \cite{FPR00} for a discussion on this topic.

It can be easily checked that the space $JL_0$ is isomorphic to a $\CC(K)$ space where $K$ is a scattered compact space with $ht(K)=2$ and sequential order 2. 
Thus, $JL_0$ has weak*-sequential dual ball with sequential order 2.

The known examples in ZFC of sequential compact spaces are all of sequential order $\leq 2$. 
However, A. Dow constructed under the assumption  $\mathfrak{b}=\mathfrak{c}$ a scattered compact space $K$  of sequential order $4$ such that the sequential  order and the scattering heigth coincide \cite{Dow}. 

\begin{CORO}
\label{Existence}
Under $\mathfrak{b}=\mathfrak{c}$, there exist Banach spaces with weak*-sequential dual balls of any sequential order $\leq 4$. 
\end{CORO}

\section*{Acknowledgements}
I would like to thank Jos\'{e} Rodr\'{i}guez for his helpful suggestions and remarks.


\begin{thebibliography}{9}
	
	
%

	\bibitem{Ba88} Z. Balogh, \textit{On compact Hausdorff spaces of countable tightness}, Proc. Amer. Math. Soc. 105 (1988), 755-764.

	\bibitem{BNP} P. Borodulin-Nadzieja, G. Plebanek, \textit{On sequential properties of Banach spaces, spaces of measures and densities}, Czech. Math. J. 60 (2010), 381--399. 
	
	\bibitem{Bou79} J. Bourgain, La propriet\'e de Radon-Nikod\'ym, Publ. Math. Univ. Pierre et Marie Curie 36 (1979).
	
	\bibitem{Brech08} C. Brech, \textit{Constru\c{c}\~{o}es gen\'ericas de espa\c{c}os de Asplund C(K)}, Ph.D. thesis, Universidade de S\~{a}o Paulo and Université Paris 7 (2008).
	
	\bibitem{Diest} J. Diestel, \textit{Sequences and series in Banach spaces}, Springer New York (1984).
	
	\bibitem{Dow} A. Dow, \textit{Sequential order under MA}, Topp. Appl. 146/147 (2005), 501--510.
	
%
	\bibitem{Edg79} G.A. Edgar, \textit{Measurability in a Banach space, II}, Ind. Univ. Math. J. 28 (1979), 559--579.
	
	
	\bibitem{FKK13} J. Ferrer, P. Koszmider, W. Kubi\'{s}, \textit{Almost disjoint families of countable sets and separable complementation properties}, J. Math. Anal. Appl. 401 (2013), 939--949.
	
	\bibitem{FPR00} R. Frankiewicz, G. Plebanek, C. Ryll-Nardzewski, \textit{Between the Lindel\"{o}f property and countable tightness}, Proc. Amer. Math. Soc. 129 (2000) 97--103
	
	\bibitem{HS80} J. Hagler, F. Sullivan, \textit{Smoothness and weak* sequential compactness}, Proc. Amer. Math. Soc. 78 (1980), 497--503.  
	
	
	\bibitem{Hay78}  R. Haydon, \textit{On dual $L^1$-spaces and injective bidual Banach spaces}, Isr. J. Math. 31 (1978), 142--152.  
	
	\bibitem{JL74} W.B. Johnson, J. Lindenstrauss, \textit{ Some remarks on weakly compactly generated {B}anach spaces}, Isr. J. Math. 17 (1974), 219--230.	
	
	\bibitem{Kap86} T. Kappeler,  \textit{Banach spaces with the condition of Mazur}, Math. Z. 191 (1986), 623--631.	
	
	 \bibitem{Ku81} K. Kunen, A compact $L$-space under CH, Topology Appl. 12 (1981), 283--287.
	  		
	\bibitem{Pol80} R. Pol, \textit{On a question of H.H. Corson and some related problems}, Fund. Math. 109 (1980)m 143--154. 
	
	\bibitem{Pli15} A. Plichko, \textit{Three sequential properties of dual Banach spaces in the weak* 
	topology}, Top. Appl. 190 (2015), 93--98.
	
	\bibitem{PY00} A. Plichko, D. Yost, \textit{Complemented and uncomplemented subspaces of Banach spaces}, Ext. Math. 15 (2000), 335--371, III Congress on Banach Spaces (Jarandilla de la Vera, 1998). 
	
	\bibitem{Yos93} D. Yost, \textit{Asplund spaces for beginners}, Acta. Univ. Carolin. Math. Phys. 34 (1993), 159--177.
%
%
%
%
\end{thebibliography}
\end{document}